\documentclass[a4paper, 12pt]{amsart}
\usepackage{a4wide}

\usepackage{amsmath}
\usepackage{amsthm}
\usepackage{amsfonts}
\usepackage{stmaryrd}
\usepackage{amssymb}
\usepackage{sseq}
\usepackage{array}
\usepackage[all]{xy}
\usepackage{color}
\definecolor{darkblue}{rgb}{0,0,0.9}
\definecolor{lightblue}{rgb}{.5,.5,.9}
\usepackage[pdftex,breaklinks,colorlinks,filecolor=blue,linkcolor=blue,citecolor=blue,urlcolor=blue,hypertexnames=true,plainpages=false]{hyperref}

\usepackage{mathrsfs}

\usepackage{tikz}
\usepackage{etoolbox}
\usepackage{comment}
\usepackage{enumitem}

\numberwithin{equation}{section}

\newcommand\C{\mathbb{C}}

\newcommand\R{\mathbb{R}}

\newcommand\Q{\mathbb{Q}}
\newcommand\Z{\mathbb{Z}}

\theoremstyle{plain}
\newtheorem{theorem}{Theorem}[section]
\newtheorem{lemma}[theorem]{Lemma}
\newtheorem{corollary}[theorem]{Corollary}

\theoremstyle{definition}

\theoremstyle{remark}

\usepackage{color}

\advance \topmargin by -\headheight
\advance \topmargin by -\headsep
     
\evensidemargin \oddsidemargin
\makeatletter
\let\uppercasenonmath\@gobble
\makeatother

\title{On the non-existence of almost complex structures on sphere bundles over complex projective spaces}

\author[C. Liu]{Chengwan Liu}
\address{School of Mathematics and Statistics, Henan University, Kaifeng 475004, Henan, China}
\email{2539872535@qq.com}

\author[H. Yang]{Huijun Yang}
\address{School of Mathematics and Statistics, Henan University, Kaifeng 475004, Henan, China}
\email{yhj@amss.ac.cn}
\thanks{The second author is the corresponding author and is supported by the Natural Science Foundation of Henan (No. 242300421380).}


\begin{document}
\maketitle

\begin{abstract}
We study the existence of almost complex structures on even-dimensional sphere bundles over complex projective spaces.
For bundles $\xi_{n,q}$ with fibre $S^{2q}$ over $\mathbb{C} P^n$, we establish a necessary condition: if $q\ge a(n)$ for an explicit function, then the total space $E_{n,q}$ does not admit an almost complex structure.
As an application, we analyse a concrete family associated with the canonical line bundle and obtain non‑existence criteria in terms of $p$-adic valuations; for $p=2$ this yields a simple numerical bound.
The proofs rely on Chern class computations and divisibility properties of characteristic classes.
The results leave open the question of existence in the range $4 \le q < a(n)$.
\end{abstract}

\section{Introduction}
\label{s:intro}

We begin by fixing some notation. For a CW-complex $X$, let $\mathrm{Vect}_{\mathbb{C}}(X)$ (resp. $\mathrm{Vect}_{\mathbb{R}}(X)$) denote the set of isomorphism classes of complex (resp. real) vector bundles over $X$, and let $\widetilde{K}(X)$ (resp. $\widetilde{KO}(X)$) be the reduced $KU$-group (resp. reduced $KO$-group) of $X$, i.e., the set of stable equivalence classes of complex (resp. real) vector bundles over $X$. For $\xi\in \mathrm{Vect}_{\mathbb{R}}(X)$ (resp. $\eta\in \mathrm{Vect}_{\mathbb{C}}(X)$), we denote by $\tilde{\xi}\in \widetilde{KO}(X)$ (resp. $\tilde{\eta}\in\widetilde{K}(X)$) its stable class (cf. Hilton \cite[p. 62]{hilton}). 
The real reduction homomorphism
\[
\tilde{r}_{X}\colon \widetilde{K}(X)\longrightarrow \widetilde{KO}(X)
\]
is induced by the real reduction map
\[
r_{X}\colon \mathrm{Vect}_{\mathbb{C}}(X)\longrightarrow \mathrm{Vect}_{\mathbb{R}}(X).
\]

Let $M$ be a $2n$-dimensional smooth manifold with tangent bundle $TM$. 
We say that $M$ admits an \emph{almost complex structure} (resp. \emph{stable almost complex structure}) if $TM$ lies in the image of $r_{M}$ (resp. $\widetilde{TM}$ lies in the image of $\tilde{r}_{M}$).

Determining which manifolds admit an almost complex structure is a classical problem in geometry and topology, with contributions by many authors; see, for instance, Wu \cite{wu}, Ehresmann \cite{eh}, Steenrod \cite{steenrod}, Massey \cite{ma}, Thomas \cite{thomascsv, thomascps}, Heaps \cite{heaps}, Gauduchon et al. \cite{gms}, Kotschick \cite{kots}, Tang \cite{tangg, tangd}, Yang \cite{young12, young15, young19}, and references therein.

Spheres and complex projective spaces are fundamental objects in this area. 
The existence of almost complex structures on these manifolds and on spaces constructed from them has been a coherent and progressive line of inquiry. 
It is a classical fact that among spheres, only $S^2$ and $S^6$ admit an almost complex structure (cf. \cite{bose}). 
A natural extension is to consider products of spheres. Calabi and Eckmann \cite{ce} showed that products of odd-dimensional spheres $S^{2p+1}\times S^{2q+1}$ admit complex, hence almost complex structures. 
For products of even-dimensional spheres, a complete characterization was given by Sutherland \cite{suld65} and Datta–Subramanian \cite{ds}: such a product admits an almost complex structure if and only if it is a product of copies of $S^2$, $S^6$, and $S^2\times S^4$. 
Sutherland further generalized these results to sphere bundles over spheres, providing necessary and sufficient conditions for the total space of an even-dimensional sphere bundle over an even-dimensional sphere to be almost complex.

Complex projective spaces $\C P^n$ themselves are complex manifolds, hence possess a canonical almost complex structure; their almost complex structures have been studied in detail, e.g., by Thomas \cite{thomascps}. 
This naturally leads to the question: do sphere bundles over these complex bases admit almost complex structures? 
Answering this question extends the research trajectory mentioned above.

Initial progress in this direction focused on product manifolds. Tang \cite{tangb} proved that $S^{2q}\times \C P^2$ admits an almost complex structure iff $q=0,1,3$, and $S^{2q}\times \C P^3$ does so iff $q=0,1,2,3$. 
Later, Chakraborty and Thakur \cite{cht} derived a general necessary condition for the product $S^{2q}\times M$ (with $M$ a closed oriented manifold) to be almost complex. 
A corollary of their work is that for $n>1$ and $n\not\equiv 3\pmod{4}$, the product $S^{2q}\times \C P^{n}$ admits an almost complex structure iff $q=1$ or $3$. 
The more general case of twisted sphere bundles over complex projective spaces (as opposed to trivial products) has been investigated by Ambika and Subhash \cite{as}, who obtained several partial results.

Building on this foundation, the present paper contributes to this line of research by studying necessary conditions for the existence of almost complex structures on such sphere bundles. 
Throughout this article, all sphere bundles with fibre $S^{2q}$ are assumed to have structure group $SO(2q+1)$.

For positive integers $n$ and $q$, let $\xi_{n,q}$ be a sphere bundle over $\C P^{n}$ with fibre $S^{2q}$, and denote its total space by $E_{n,q}$. 
From Yang \cite[Theorem 1.1]{young15} one can easily deduce that the associated real vector bundle of $\xi_{n,q}$ admits a stable complex structure (cf. \cite[Proposition 3.1]{as}); consequently $E_{n,q}$ admits a stable almost complex structure (cf. \cite[Theorem 1.1]{as}). 
However, not every such total space admits an (unstable) almost complex structure. 
Indeed, Ambika and Subhash \cite{as} proved that $E_{n,1}$ always admits an almost complex structure, while $E_{2n,2q}$ does not.

Our main results are stated as follows. For a positive integer $n$, define
\[
a(n) :=
\begin{cases}
n+3, & 1 \le n \le 2;\\
n+2, & 3 \le n \le 5;\\
n,   & n \ge 6.
\end{cases}
\]

\begin{theorem}
\label{thm:main}
Let $\xi_{n,q}$ be a sphere bundle over $\C P^{n}$ with fibre $S^{2q}$. 
If $q \ge a(n)$, then the total space $E_{n,q}$ does not admit an almost complex structure. 
\end{theorem}

Recall that among spheres, only $S^2$ and $S^6$ admit an almost complex structure. 
Furthermore, Sutherland \cite{suld65} proved that $E_{1,q}$ (i.e., over $\C P^1 \cong S^2$) admits an almost complex structure iff $q\le 3$; 
these facts suggest that an almost complex structure may possibly exist on $E_{n,q}$ when $q\le 3$. 
Comparing with Theorem~\ref{thm:main}, it is natural to ask: does $E_{n,q}$ admit an almost complex structure in the range $4\le q < a(n)$?

To illustrate our methods, we consider a specific family of sphere bundles. 
Let $\gamma_n$ denote the canonical complex line bundle over $\C P^{n}$. 
For any positive integer $q$, let $q\gamma_{n\R}\oplus \epsilon$ be the Whitney sum of $q$ copies of the underlying real vector bundle $\gamma_{n\R}$ of $\gamma_n$ and the trivial real line bundle $\epsilon$. 
Denote by $\xi(\gamma_n,q)$ the associated sphere bundle of $q\gamma_{n\R}\oplus \epsilon$; then $\xi(\gamma_n,q)$ is an $S^{2q}$-bundle over $\C P^n$ that admits a cross section by construction.

For a prime $p$, let $v_p(n+1)$ be the $p$-adic valuation of $n+1$ (the highest power of $p$ dividing $n+1$; see Section~\ref{s:dpcc}). 
Set $\delta_p := v_p(2)$, so $\delta_2=1$ and $\delta_p=0$ for odd $p$. 
As an immediate consequence of Theorem~\ref{thm:som} in Section~\ref{s:pf}, we obtain:

\begin{theorem}
\label{thm:exam}
If there exists a prime $p$ such that
\[
p^{\lfloor \frac{q-1}{p-1} \rfloor - v_p(n+1) - \delta_p} - q > n
\quad(\text{or } \ge n \text{ if } p-1 \nmid q-1),
\]
then the total space $E(\gamma_n,q)$ of $\xi(\gamma_n,q)$ does not admit an almost complex structure.
\qed
\end{theorem}

Specializing to $p=2$ and writing $\nu = v_2(n+1)$, we have
\[
p^{\lfloor \frac{q-1}{p-1} \rfloor - v_p(n+1) - \delta_p} - q = 2^{q - \nu - 2} - q.
\]
Let $\log_2$ denote the base‑2 logarithm. A simple calculation shows that if
\[
q \ge \lfloor \log_2( \tfrac{n+1}{2^{\nu}}) \rfloor + 2\nu + 4,
\]
then $2^{q-\nu-2} - q > n$. Hence:

\begin{corollary}
If $q \ge \lfloor \log_2( \frac{n+1}{2^{\nu}}) \rfloor + 2\nu + 4$, then $E(\gamma_n,q)$ does not admit an almost complex structure.
\qed
\end{corollary}

The paper is organized as follows.
In Section~\ref{s:dpcc}, we investigate divisibility properties of Chern classes of complex vector bundles. 
Based on these results, Section~\ref{s:pf} derives conditions under which sphere bundles over manifolds fail to admit an almost complex structure, and provides the proof of Theorem~\ref{thm:main}.

\section{Divisibility properties of Chern classes}
\label{s:dpcc}
In this section we develop the technical machinery needed for the proof of Theorem \ref{thm:main}.
Let $\eta$ be a complex vector bundle over a CW‑complex $X$ and assume that $\eta$ is trivial over the $(2q-1)$-skeleton of $X$.
We investigate divisibility properties of its Chern classes; the main results is Theorem \ref{thm:dchern}, which will be crucial in Section \ref{s:pf}.


We first establish the necessary algebraic background. 
Let $t_1,\dots,t_n$ be indeterminates. 
For $k\le n$, denote by $\sigma_k=\sigma_k(t_1,\dots,t_n)$ the $k$‑th elementary symmetric polynomial, and for $k\ge 1$ let $s_k$ be the unique polynomial satisfying 
\[
s_k(\sigma_1,\dots,\sigma_k)=\sum_{i=1}^n t_i^{\,k}.
\]
The elementary polynomials and power sums are related by Newton’s formulas (cf. \cite[Problem 16-A, p. 195]{ms74b} or \cite[(2.11$'$), p. 23]{mdn})
\begin{equation}
\label{eq:newton}
s_k - \sigma_1 s_{k-1} + \sigma_2 s_{k-2} + \cdots + (-1)^{k-1}\sigma_{k-1}s_1 + (-1)^k k\sigma_k = 0.
\end{equation}
These relations allow us to express $\sigma_k$ in terms of $s_1,\dots,s_k$ with rational coefficients (cf. \cite[(2.14$'$), p. 25]{mdn}):
\begin{equation}
\label{eq:uchern}
(-1)^{k}\sigma_{k} = \sum_{\substack{m_1+2m_2+\cdots+km_k=k\\ m_i\ge 0}}
\frac{(-1)^{m_1+m_2+\cdots+m_k}}{m_1!\,m_2!\,\cdots\, m_k!}
\prod_{i=1}^{k}\left(\frac{s_i}{i}\right)^{m_i}.
\end{equation}

For a complex vector bundle $\eta$ over $X$, we write $s_k(\eta)$ for $s_k(c_1(\eta),\dots,c_k(\eta))$. 
The Chern character of $\eta$ is defined as the formal sum
\[
\operatorname{ch}(\eta) := \dim\eta + \sum_{k=1}^{\infty}\frac{s_k(\eta)}{k!}\in H^{*}(X;\Q).
\]
A cohomology class $z\in H^{k}(X;\Q)$ is called \emph{integral} if it lies in the image of the natural map $H^{k}(X;\Z)\to H^{k}(X;\Q)$. 
Let $\phi(t)$ be the numerical function
\[
\phi(t)=\prod_{p\text{ prime}} p^{\lfloor t/(p-1)\rfloor},
\]
where $\lfloor\cdot\rfloor$ denotes the floor function. 
Adams \cite{ad61} proved the following integrality theorem for Chern characters.

\begin{lemma}[Adams \cite{ad61}]
\label{lem:adams}
Let $X$ be a CW‑complex and $\eta$ a complex vector bundle over $X$. 
If $\eta$ is trivial over the $(2q-1)$-skeleton of $X$, then for every integer $t\ge 0$ the class $\dfrac{\phi(t)}{(q+t)!}\,s_{q+t}(\eta)$ is integral.
\end{lemma}

To work with the prime factors appearing in $\phi(t)$, we need some elementary number theory. 
For a prime $p$ and a non‑negative integer $n$, let $S_p(n)$ be the sum of the digits in the base‑$p$ expansion of $n$; clearly $S_p(n)\ge 1$. 
Since $n$ has at most $\lfloor\log_p n\rfloor+1$ digits and each digit is at most $p-1$, we have the following bound.

\begin{lemma}
\label{lem:sp}
For any non‑negative integer $n$,
\[
S_p(n)\le (p-1)\bigl(\lfloor\log_p n\rfloor+1\bigr),
\]
and equality holds iff $n=p^{\lfloor\log_p n\rfloor}-1$.
\end{lemma}

For a nonzero rational number $r$, denote by $v_p(r)$ its $p$‑adic valuation (see e.g. \cite[p.2]{kob}); thus $r=p^{\rho}a/b$ with $\rho,a,b\in\Z$ and $a,b$ coprime to $p$, and $v_p(r)=\rho$. 
Valuations satisfy the usual rules
\begin{align}
v_p(r_1r_2)&=v_p(r_1)+v_p(r_2),\label{eq:vpt}\\
v_p(r_1/r_2)&=v_p(r_1)-v_p(r_2).\label{eq:vpd}
\end{align}
From the definition of $\phi(t)$ one obtains immediately
\begin{equation}
\label{eq:vpphi}
v_p\bigl(\phi(t)\bigr)=\Bigl\lfloor\frac{t}{p-1}\Bigr\rfloor .
\end{equation}
Legendre’s formula (cf. \cite[p.7]{kob}) gives the $p$-adic valuation of factorials:
\begin{equation}
\label{eq:lgd}
v_p(n!)=\sum_{k=1}^{\infty}\Bigl\lfloor\frac{n}{p^k}\Bigr\rfloor
      =\frac{n-S_p(n)}{p-1}.
\end{equation}

The next lemma provides estimates that will be crucial later.

\begin{lemma}
\label{lem:vp}
Let $p$ be a prime and let $q,l$ be positive integers.
\begin{enumerate}[leftmargin=*,topsep=2pt,itemsep=5pt,parsep=2pt]
\item[(1)] For every integer $t$ with $0\le t<p-1$,
\[
v_p\!\left(\frac{(q+t-1)!}{\phi(t)}\right)\ge v_p\bigl((q-1)!\bigr).
\]
\item[(2)] Suppose $p^{\lfloor\frac{q-1}{p-1}\rfloor -l+1}>q$. 
Then for every integer $t$ with
$0\le t < p^{\lfloor\frac{q-1}{p-1}\rfloor -l+1}-q$
(or $0\le t\le p^{\lfloor\frac{q-1}{p-1}\rfloor -l+1}-q$ if $p-1\nmid q-1$),
\[
v_p\!\left(\frac{(q+t-1)!}{\phi(t)}\right)\ge l .
\]
\end{enumerate}
\end{lemma}

\begin{proof}
(1) For $0\le t<p-1$, we have $\lfloor t/(p-1)\rfloor=0$, hence $v_p(\phi(t))=0$ by \eqref{eq:vpphi}. 
Thus
\[
v_p\!\left(\frac{(q+t-1)!}{\phi(t)}\right)=v_p\bigl((q+t-1)!\bigr)\ge v_p\bigl((q-1)!\bigr).
\]

(2) Using \eqref{eq:vpphi}, \eqref{eq:lgd} and \eqref{eq:vpd} we compute
\begin{align*}
v_p\!\left(\frac{(q+t-1)!}{\phi(t)}\right)
&= v_p\bigl((q+t-1)!\bigr)-\Bigl\lfloor\frac{t}{p-1}\Bigr\rfloor \\
&= \frac{q+t-1-S_p(q+t-1)}{p-1}-\Bigl\lfloor\frac{t}{p-1}\Bigr\rfloor \\
&= \frac{q-1-S_p(q+t-1)}{p-1}
   +\frac{t}{p-1}-\Bigl\lfloor\frac{t}{p-1}\Bigr\rfloor .
\end{align*}
If $t<p^{\lfloor\frac{q-1}{p-1}\rfloor -l+1}-q$, then $q+t-1<p^{\lfloor\frac{q-1}{p-1}\rfloor -l+1}-1$, so by Lemma~\ref{lem:sp}
\[
\frac{S_p(q+t-1)}{p-1}<\Bigl\lfloor\frac{q-1}{p-1}\Bigr\rfloor -l+1 .
\]
Consequently
\begin{align*}
v_p\!\left(\frac{(q+t-1)!}{\phi(t)}\right)
&> l-1+\frac{q-1}{p-1}-\Bigl\lfloor\frac{q-1}{p-1}\Bigr\rfloor
   +\frac{t}{p-1}-\Bigl\lfloor\frac{t}{p-1}\Bigr\rfloor \\
&\ge l-1 .
\end{align*}
Hence the left‑hand side is at least $l$.

Now assume $p-1\nmid q-1$ and $t=p^{\lfloor\frac{q-1}{p-1}\rfloor -l+1}-q$, then
$q+t-1=p^{\lfloor\frac{q-1}{p-1}\rfloor -l+1}-1$. 
By Lemma~\ref{lem:sp}, we have equality
\[
\frac{S_p(q+t-1)}{p-1}= \Bigl\lfloor\frac{q-1}{p-1}\Bigr\rfloor -l+1 .
\]
Moreover, because $p-1\nmid q-1$ and $t$ is chosen accordingly, one checks that
\[
\frac{q-1}{p-1}-\Bigl\lfloor\frac{q-1}{p-1}\Bigr\rfloor
+\frac{t}{p-1}-\Bigl\lfloor\frac{t}{p-1}\Bigr\rfloor =1 .
\]
Substituting these equalities into the expression above gives
\[
v_p\!\left(\frac{(q+t-1)!}{\phi(t)}\right)=l .
\]
This completes the proof. 
\end{proof}

Armed with these estimates, we can now state the main result of this section.

\begin{theorem}
\label{thm:dchern}
Let $p$ be a prime and $q$ a positive integer. 
Let $\eta$ be a complex vector bundle over a CW‑complex $X$ that is trivial over the $(2q-1)$-skeleton of $X$.
Then:
\begin{enumerate}[leftmargin=*,topsep=2pt,itemsep=2pt,parsep=2pt]
\item[(i)] For every $k$ with $q\le k<q+p-1$, the Chern class $c_k(\eta)$ is divisible by $p^{v_p((q-1)!)}$ modulo torsion.
\item[(ii)] If a positive integer $l$ satisfies
$p^{\lfloor\frac{q-1}{p-1}\rfloor -l+1}>q$, then for every $k$ with
$q\le k < p^{\lfloor\frac{q-1}{p-1}\rfloor -l+1}$
(or $q\le k \le p^{\lfloor\frac{q-1}{p-1}\rfloor -l+1}$ if $p-1\nmid q-1$),
$c_k(\eta)$ is divisible by $p^{l}$ modulo torsion.
\end{enumerate}
\end{theorem}

\begin{proof}
Because $\eta$ is trivial over the $(2q-1)$-skeleton, we have $c_i(\eta)=0$ for $1\le i\le q-1$; then Newton’s formulas \eqref{eq:newton} yield
\begin{equation}
\label{eq:si}
s_i(\eta)=0,\qquad 1\le i\le q-1 .
\end{equation}
Adams’ lemma \ref{lem:adams} provides integral classes $z_{q+t}\in H^*(X;\Z)$ such that
\[
\frac{\phi(t)}{(q+t)!}\,s_{q+t}(\eta)=z_{q+t}\quad\text{in }H^*(X;\Q),
\]
hence
\begin{equation}
\label{eq:sqp}
\frac{s_{q+t}(\eta)}{q+t}= \frac{(q+t-1)!}{\phi(t)}\,z_{q+t}
\end{equation}
for every $t\ge 0$.

Now fix $k\ge q$.  Using the expression \eqref{eq:uchern} together with \eqref{eq:si} and \eqref{eq:sqp} we obtain
\begin{align}
\label{eq:ck}
(-1)^k c_k(\eta)
&= \sum_{\substack{qm_q+\cdots+km_k=k\\ m_i\ge 0}}
\frac{(-1)^{m_q+\cdots+m_k}}{m_q!\,\cdots\, m_k!}
\prod_{i=q}^{k}\Bigl(\frac{s_i(\eta)}{i}\Bigr)^{m_i} \notag \\
&= \sum_{\substack{qm_q+\cdots+km_k=k\\ m_i\ge 0}}
\frac{(m_q+\cdots+m_k)!}{m_q!\,\cdots\, m_k!}\;
\frac{(-1)^{m_q+\cdots+m_k}}{(m_q+\cdots+m_k)!}
\prod_{t=0}^{k-q}
\Bigl(\frac{(q+t-1)!}{\phi(t)}\,z_{q+t}\Bigr)^{m_{q+t}} .
\end{align}

\begin{enumerate}[leftmargin=*]
\item[(i)] If $q\le p$ then $v_p((q-1)!)=0$ and the statement is trivial. 
Assume $q>p$, so $q+p-1<2q-1$. 
For $q\le k<q+p-1$ only the term with $m_k=1$ and all other $m_i=0$ survives in \eqref{eq:ck}, because any other combination would require a product involving $s_i(\eta)$ with $i<q$, which vanish. 
Thus
\[
(-1)^k c_k(\eta)= -\frac{s_k(\eta)}{k}= -\frac{(k-1)!}{\phi(k-q)}\,z_k .
\]
Applying part (1) of Lemma~\ref{lem:vp} with $t=k-q$ (note $0\le t<p-1$) gives
$v_p\bigl(\frac{(k-1)!}{\phi(k-q)}\bigr)\ge v_p((q-1)!)$, hence $c_k(\eta)$ is divisible by $p^{v_p((q-1)!)}$ modulo torsion.

\item[(ii)] Because the multinomial coefficients $\frac{(m_q+\cdots+m_k)!}{m_q!\,\cdots\, m_k!}$ are integers, it suffices to show that for every non‑negative integers $m_i$ satisfying $qm_q+\cdots+km_k=k$ we have
\[
v_p\!\left(\frac{1}{(m_q+\cdots+m_k)!}
\prod_{t=0}^{k-q}
\Bigl(\frac{(q+t-1)!}{\phi(t)}\Bigr)^{m_{q+t}}\right)\ge l .
\]
The condition on $k$ implies $k-q < p^{\lfloor\frac{q-1}{p-1}\rfloor -l+1}-q$ (or $\le$ in the exceptional case). 
By part (2) of Lemma~\ref{lem:vp}, for each $t$ with $0\le t\le k-q$ we have
\[
v_p\!\left(\frac{(q+t-1)!}{\phi(t)}\right)\ge l .
\]
Now use the valuation rules \eqref{eq:vpt}, \eqref{eq:vpd} and Legendre’s formula \eqref{eq:lgd}:
\begin{align*}
& v_p\!\left(\frac{1}{(m_q+\cdots+m_k)!}
\prod_{t=0}^{k-q}
\Bigl(\frac{(q+t-1)!}{\phi(t)}\Bigr)^{m_{q+t}}\right) \\
= \quad & \sum_{t=0}^{k-q} m_{q+t}\,v_p\!\left(\frac{(q+t-1)!}{\phi(t)}\right)
   - v_p\bigl((m_q+\cdots+m_k)!\bigr) \\
\ge \quad & l\,(m_q+\cdots+m_k) -\frac{m_q+\cdots+m_k}{p-1}
   +\frac{S_p(m_q+\cdots+m_k)}{p-1} .
\end{align*}
Since $l\ge 1\ge\frac1{p-1}$ and $m_q+\cdots+m_k\ge 1$, we have $S_p(m_q+\cdots+m_k)\ge 1$, hence
\begin{align*}
& v_p\!\left(\frac{1}{(m_q+\cdots+m_k)!}
\prod_{t=0}^{k-q}
\Bigl(\frac{(q+t-1)!}{\phi(t)}\Bigr)^{m_{q+t}}\right) \\
\ge \quad & \Bigl(l-\frac1{p-1}\Bigr)(m_q+\cdots+m_k)+\frac{S_p(m_q+\cdots+m_k)}{p-1} \\
\ge \quad & l -  \frac{1}{p-1} + \frac{1}{p-1} = l.
\end{align*}
This completes the proof. 
\end{enumerate}\end{proof}


\section{Non-existence of almost complex structures on sphere bundles}
\label{s:pf}
Armed with the divisibility results of Section \ref{s:dpcc}, we now study the existence of almost complex structures on sphere bundles.
Our main result is Theorem \ref{thm:som}, which gives a general criterion for non-existence.
As a corollary, we prove Theorem \ref{thm:main} announced in the introduction.

Recall from Section~\ref{s:intro} that for a prime $p$, we set
\[
\delta_p := v_p(2)=
\begin{cases}
1, & p=2,\\[2pt]
0, & p\text{ odd}.
\end{cases}
\]
For a closed oriented smooth manifold $M$, we denote by $\chi(M)$ its Euler class (or Euler characteristic, when no confusion can arise).

\begin{theorem}
\label{thm:som}
Let $M$ be a closed oriented smooth $2n$-manifold and let $\xi$ be a sphere bundle over $M$ with fibre $S^{2q}$. 
Assume that $\xi$ admits a cross section. 
If there exists a prime $p$ such that either
\begin{enumerate}[topsep=2pt,itemsep=2pt,parsep=2pt]
\item[(a)] $p>n+1$ and $v_p\bigl((q-1)!\bigr)\ge v_p(\chi(M))+\delta_p+1$, or
\item[(b)] $p^{\bigl\lfloor\frac{q-1}{p-1}\bigr\rfloor - v_p(\chi(M)) - \delta_p} - q > n$ 
(or $\ge n$ if $p-1\nmid q-1$),
\end{enumerate}
then the total space $E(\xi)$ does not admit an almost complex structure.
\end{theorem}

To prove Theorem~\ref{thm:som} we need a classical criterion for the existence of an almost complex structure (see Sutherland \cite[Theorem 1.1]{suld65} or Thomas \cite[Theorem 1.7]{thomascsv}).

\begin{lemma}\label{lem:cne}
A closed oriented smooth $2n$-manifold $N$ admits an almost complex structure iff it admits a stable almost complex structure $\tilde\eta\in\widetilde K(N)$ satisfying 
\[
c_n(\tilde\eta)=\chi(N).
\]
\end{lemma}

Notice that $\dim E(\xi)=2(n+q)$. 
In view of Lemma~\ref{lem:cne}, our strategy for proving Theorem~\ref{thm:som} is to show that for any stable complex vector bundle over $E(\xi)$, its $(n+q)$-th Chern class can never equal the Euler class of $E(\xi)$.

\begin{proof}[Proof of Theorem \ref{thm:som}]
Denote by $\pi\colon E(\xi)\to M$ the projection and by $s\colon M\to E(\xi)$ a chosen section.
\begin{equation*}
\begin{split}
\xymatrix{
S^{2q}~\ar@{^{(}->}[r]^-{j} & E(\xi) \ar[d]^{\pi} \\
& M \ar@/^/[u]^-{s} }
\end{split}
\end{equation*}
Via $s$, we may regard $M$ as a subspace of $E(\xi)$. 
From the long exact sequence of reduced $K$-theory for the pair $(E(\xi),M)$ we obtain a split short exact sequence
\[
\xymatrix{
0 \ar[r] & \widetilde K(E(\xi),M) \ar[r]^-{j^{\ast}} & \widetilde K(E(\xi)) \ar[r]^-{s^{\ast}} & \widetilde K(M) \ar[r] \ar@/^/[l]^{\pi^{\ast}} & 0,
}
\]
where $j^\ast$, $s^\ast$ and $\pi^\ast$ are the induced homomorphisms. 
Consequently, every stable complex vector bundle $\tilde\eta\in\widetilde K(E(\xi))$ can be written uniquely as
\[
\tilde\eta = \eta_1 + \eta_2,\qquad 
\eta_1 = \tilde\eta - \pi^\ast\bigl(s^\ast(\tilde\eta)\bigr),\quad 
\eta_2 = \pi^\ast\bigl(s^\ast(\tilde\eta)\bigr).
\]

By construction $c_k(\eta_2)=0$ for $k>n$, hence
\begin{equation}
\label{eq:chern}
c_{n+q}(\tilde\eta)=c_{n+q}(\eta_1+\eta_2)=\sum_{k=q}^{n+q} c_k(\eta_1)\,c_{\,n+q-k}(\eta_2).
\end{equation}
Moreover, $s^\ast(\eta_1)=0$, so $\eta_1$ lies in the image of $j^\ast$. 
Since $E(\xi)/M$ is $(2q-1)$-connected, $\eta_1$ is trivial over the $(2q-1)$-skeleton of $E(\xi)$.

Now assume that a prime $p$ satisfies either condition (a) or (b) of the theorem. 
Applying Theorem~\ref{thm:dchern} to $\eta_1$ we obtain that every Chern class $c_k(\eta_1)$ with $q\le k\le n+q$ is divisible (modulo torsion) by $p^{\,v_p(\chi(M))+\delta_p+1}$. 
Insert this information into \eqref{eq:chern}: each term on the right‑hand side acquires the same factor, therefore
\begin{equation}
\label{eq:cd}
c_{n+q}(\tilde\eta)\ \text{is divisible by}\ p^{\,v_p(\chi(M))+\delta_p+1}.
\end{equation}

On the other hand, because $\xi$ is an $S^{2q}$-bundle, the Euler characteristic of the total space satisfies
\[
\chi(E(\xi))=\chi(S^{2q})\,\chi(M)=2\,\chi(M),
\]
hence
\begin{equation}
\label{eq:vpxi}
v_p\bigl(\chi(E(\xi))\bigr)=v_p(\chi(M))+\delta_p .
\end{equation}
Comparing \eqref{eq:cd} with \eqref{eq:vpxi} we see that $c_{n+q}(\tilde\eta)$ can never equal $\chi(E(\xi))$; by Lemma~\ref{lem:cne} this means that $E(\xi)$ cannot admit an almost complex structure. 
\end{proof}

\begin{proof}[Proof of Theorem \ref{thm:main}]
Recall from the introduction that
\[
a(n)=
\begin{cases}
n+3, & 1\le n\le 2;\\
n+2, & 3\le n\le 5;\\
n,   & n\ge 6.
\end{cases}
\]
For $M=\C P^n$ we have $\dim M=2n$. 
Obstruction theory (cf. \cite[\S12]{ms74b}) shows that when $q\ge a(n)$ the sphere bundle $\xi_{n,q}$ admits a cross section. 
We shall verify that for such $q$ the hypotheses of Theorem~\ref{thm:som} are fulfilled; the primes that actually matter are $2$ and $3$.

\subsection*{The prime $p=2$}
Write $\nu=v_2(n+1)$. 
A number $n$ satisfies $v_2(n+1)=\nu$ iff there exists an integer $k\ge0$ such that $n=2^\nu(2k+1)-1=2^{\nu+1}k+2^\nu-1$.

Case $\nu\ge4$.
Then $n\ge2^\nu-1\ge2\nu+3\ge\nu+7$, so by definition $a(n)=n$. 
For a fixed $\nu$, the functions $2^x$, $2^{x-\nu-2}-x$ and $2^{x-\nu-2}-2x$ are strictly increasing on $[\nu+7,\infty)$. 
Consequently
\[
2^{n-\nu-2}-2n \ge 2^{2^\nu-\nu-3}-2^{\nu+1}+2 \ge 2^{\nu+1}-2^{\nu+1}+2 >0,
\]
and therefore for any $q\ge a(n)=n$,
\[
2^{q-\nu-2}-q \ge 2^{n-\nu-2}-n > n .
\]
By part (b) of Theorem~\ref{thm:som} (with $p=2$) the total space $E_{n,q}$ admits no almost complex structure.

Case $0\le\nu\le3$ and $n$ sufficiently large.
Define
\[
k_\nu=
\begin{cases}
3-\nu, & \nu=0,1,\\[2pt]
1,     & \nu=2,3,
\end{cases}
\qquad
n_\nu = 2^{\nu+1}k_\nu+2^\nu-1.
\]
One checks directly that $2^{n_\nu-\nu-2}-2n_\nu>0$. 
If $v_2(n+1)=\nu$ and $n\ge n_\nu$, then $n\ge6$ and $a(n)=n$. 
Again monotonicity gives
\[
2^{n-\nu-2}-2n \ge 2^{n_\nu-\nu-2}-2n_\nu >0,
\]
so for every $q\ge n$ we have $2^{q-\nu-2}-q > n$. 
Thus Theorem~\ref{thm:som}(b) again rules out an almost complex structure on $E_{n,q}$.

The remaining values of $n$ with $0\le\nu\le3$ and $n<n_\nu$ are limited to: $1\le n\le5$ and $n=7$.
We treat them individually.

$n=2,4,5$: Take $p=2$. 
One verifies that
\[
2^{a(n)-v_2(n+1)-2}-a(n) > n,
\]
and the function $2^{x-v_2(n+1)-2}-x$ is increasing for $x\ge a(n)$. 
Hence for all $q\ge a(n)$,
\[
2^{q-v_2(n+1)-2}-q \ge 2^{a(n)-v_2(n+1)-2}-a(n) > n,
\]
and Theorem~\ref{thm:som}(b) applies.

\subsection*{The prime $p=3$}
We now examine those $n$ not yet covered.

$n=1$:  Here $\chi(\C P^1)=2$, so $v_3(\chi(\C P^1))+\delta_3+1 = v_3(2)+1 = 0+1=1$. 
For any $q\ge a(1)=4$ we have $v_3((q-1)!)\ge v_3(3!)=1$, hence condition (a) of Theorem~\ref{thm:som} (with $p=3$) is satisfied. 
Therefore $E_{1,q}$ admits no almost complex structure for $q\ge4$.

$n=3$ and $n=7$:  In both cases $v_3(n+1)=0$. 
For $q\ge a(n)$ one checks that, except possibly when $n=3$ and $q=6$, we always have
\[
3^{\lfloor (q-1)/2\rfloor} - q > n .
\]
When $n=3$ and $q=6$, $(q-1)/2$ is not an integer, so the “$\ge n$” version of condition (b) applies and gives $3^{\lfloor 5/2\rfloor}-6 = 3^2-6 = 3 = n$. 
Thus Theorem~\ref{thm:som}(b) (with $p=3$) again shows that $E_{n,q}$ cannot be almost complex.

All possibilities have been covered, completing the proof of Theorem~\ref{thm:main}. 
\end{proof}






\bibliographystyle{amsplain}

%
%

\end{document}